\documentclass[11pt]{article}
\usepackage{amsthm,amsmath,amssymb}
\usepackage{url}
\usepackage[margin=1in]{geometry} 
\usepackage{caption}
\usepackage{subcaption}
\usepackage{graphicx}
\usepackage{epstopdf}
\usepackage{color}
\usepackage{framed}
\usepackage{lscape}
\usepackage{rotating}
\usepackage{algorithm}
\usepackage[noend]{algpseudocode}
\usepackage{grffile}
\usepackage{multirow}
\usepackage{xcolor}
\usepackage{comment}
\usepackage{xr}

\usepackage[colorlinks=true,linkcolor=black,citecolor=black,urlcolor=black]{hyperref}

\newtheorem{thm}{Theorem}[section]
\newtheorem{definition}[thm]{Definition}
\newtheorem{lem}[thm]{Lemma}

\theoremstyle{thm}

\theoremstyle{definition}
\newtheorem{rmk}{Remark}[section]

\newcommand{\TT}{\mathcal{T}}

\newcommand{\R}{\mathbb{R}}

\newcommand{\Z}{\mathbb{Z}}

\newcommand{\x}{\mathbf{x}}

\renewcommand{\sp}{\text{span}}
\newcommand{\0}{\mathbf{0}}

\newcommand{\RR}{\mathcal{R}}

\newcommand{\supp}{\mathrm{supp}}

\begin{document}
\title{A note on identifiability conditions in confirmatory factor analysis}
\author{William Leeb\thanks{
    School of Mathematics, University of Minnesota}
}
\date{}
\maketitle
\abstract{
Recently, Chen, Li and Zhang established conditions characterizing asymptotic identifiability of latent factors in confirmatory factor analysis. We give an elementary proof showing that a similar characterization holds non-asymptotically, and prove a related result for identifiability of factor loadings.
}

\section{Introduction}

We consider the problem of recovering a low-rank factorization of a large matrix $M$. We assume that $M = \Theta A^T$, where $\Theta$ and $A$ each have $K$ columns for a known value of $K$ much smaller than either dimension of $M$. We think of the rows of $M$ as labeling members of a population, and the columns of $M$ as labeling attributes. Following the language of factor analysis, we describe the columns of $\Theta$ as ``latent factors'', and the columns of $A$ as ``factor loadings''. For general background on factor analysis, we refer the reader to \cite{anderson1962introduction, bartholomew2011latent, jolliffe2002principal}, and references contained therein.

The factorization $M = \Theta A^T$ is not unique since for any $K$-by-$K$ invertible matrix $B$, $M = (\Theta B) (A B^{-T})^T$. In confirmatory factor analysis we are given additional  ``side information'' that specifies the support of each column of $A$. More precisely, we have a binary matrix $Q$ of the same dimensions as $A$, where $Q_{jk} = 0$ implies $A_{jk}=0$. $Q$ is referred to as a ``design matrix''. The question then arises as to what conditions on $Q$ are enough to ensure uniqueness of $M$'s factorization, up to a rescaling of the columns of $\Theta$ and $A$.

The recent paper \cite{chen2019structured} provides necessary and sufficient conditions on the matrix $Q$ under which individual columns of $\Theta$ (the latent factors) are asymptotically determined up to rescaling, or \emph{identifiable}, under certain assumptions on $\Theta$ and $A$. In this note, we show that a similar characterization applies as well in a non-asymptotic setting. We also provide an elementary proof of a similar characterization of the identifiability of $A$'s columns (the factor loadings), a question which has also attracted interest \cite{reiersol1950identifiability, koopmans1950identification, anderson1956statistical, shapiro1985identifiability,  bing2019adaptive}.

The remainder of this note is structured as follows. In Section \ref{sec-model}, we describe the precise model and terminology we will be using throughout. In Section \ref{sec-main}, we state and prove the main results, namely characterizing when the columns of $\Theta$ and $A$ are identifiable within our model. In Section \ref{sec-additional}, we compare our results to those in \cite{chen2019structured}.

\section{Definitions and model description}
\label{sec-model}

For positive integers $K$, $N$ and $J$, let $\Theta_1,\dots,\Theta_K$ be vectors in $\R^N$ and let $A_1,\dots,A_K$ be vectors in $\R^J$; and define the matrices $\Theta = [\Theta_1,\dots,\Theta_K] \in \R^{N \times K}$ and $A = [A_1,\dots,A_K] \in \R^{J \times K}$. Define the $N \times J$ matrix $M = \Theta A^T$. Additionally, let $Q \in \{0,1\}^{J \times K}$ be a binary matrix with columns $Q_1,\dots,Q_K$.

\begin{rmk}
\label{rmk-asymptotic}
The main results of this paper, namely Theorems \ref{id-theta} and \ref{id-a}, do not depend on the precise values of $N$ and $J$ (so long as both are sufficiently big), and apply equally well in the doubly-asymptotic setting where $\Theta$ and $A$ both have infinitely many rows; that is, we may view $\Theta_1,\dots,\Theta_K$ and $A_1,\dots,A_K$ as functions on $\Z_+$, the set of positive integers. This is a common assumption in work on factor analysis, and appears in the work \cite{chen2019structured} on which the present work is based. Doubly-asymptotic models for related problems also appear in, for example, \cite{bai2012statistical}, \cite{stock2002forecasting}, \cite{dobriban2019deterministic}, \cite{johnstone2001distribution}, \cite{donoho2018optimal}, \cite{leeb2021optimal}, \cite{dobriban2020optimal}, \cite{hong2018asymptotic}, to give only a partial list. This setting provides a formalization of the ``high-dimensional, large-sample'' regime where the number of parameters (in our setting, $J$) is comparable to the number of observations (in our setting, $N$). Because certain random quantities converge to deterministic limits as $N$ and $J$ grow, this asymptotic model often provides a convenient framework for analyzing statistical phenomena.
\end{rmk}

\begin{definition}
For a subset $S \subset \{1,\dots,K\}$, we define $\RR(S) \subset \Z_+$ to be the set of indices $j \in \Z_+$ such that $Q_k(j) = 1$ whenever $k \in S$, and $Q_k(j)=0$ whenever $k \notin S$.
\end{definition}

We introduce some additional notation. For a vector $\x \in \R^I$ and a subset $\RR \subset \{1,\dots,I\}$, we will denote by $\x(\RR)$ the restriction of $\x$ to $\RR$. If in addition $S \subset \{1,\dots,K\}$ and $B$ is a $I \times K$ matrix, we will denote by $B_{[\RR,S]}$ the submatrix of $B$ with rows from $\RR$ and columns from $S$. We will also use colon notation to denote ranges of indices; for example, $\x(1:n)$ denotes the subvector consisting of the first $n$ entries of $\x$; $B_{[:,S]}$ denotes the submatrix of $B$ with column indices in $S$; and so forth.

We now describe our assumptions on $\Theta$, $A$ and $Q$.

\vspace{1em}

\noindent
\textbf{Model Assumptions}
\begin{enumerate}

\item
\label{Theta-incoherence}
The columns of $\Theta$ are linearly independent.

\item
\label{A-incoherence}

If $S \subset \{1,\dots,K\}$ and $\RR(S)$ is non-empty, then the columns of the submatrix $A_{[\RR(S),S]}$ are linearly independent.

\item
\label{constraint-condition}
For any $k = 1,\dots,K$, $A_k(j)=0$ whenever $Q_k(j) = 0$.

\item
\label{C-bound}
There is a constant $C > 0$ such that
\begin{align}
\sup_{1 \le k \le K \atop i \ge 1} |\Theta_k(i)| < C, \quad
\sup_{1 \le k \le K \atop j \ge 1} |A_k(j)| < C.
\end{align}
\end{enumerate}

With the model described, we now define \emph{identifiability} of the latent factors and factor loadings.

\begin{definition}
The latent factor $\Theta_k$ is \emph{identifiable} if for any decomposition $M = \widetilde \Theta \widetilde A^T$ satisfying assumptions \ref{Theta-incoherence} -- \ref{C-bound}, $\Theta_k$ and $\widetilde \Theta_k$ are linearly dependent (that is, one is a scalar multiple of the other). Similarly, the factor loading $A_k$ is \emph{identifiable} if for any decomposition $M = \widetilde \Theta \widetilde A^T$ satisfying assumptions \ref{Theta-incoherence} -- \ref{C-bound}, $A_k$ and $\widetilde A_k$ are linearly dependent.
\end{definition}

\begin{rmk}
\label{rmk-id}
\cite{chen2019structured} defines identifiability of $\Theta_k$ to mean that the angle between $\Theta_k(1:N)$ and $\widetilde \Theta_k(1:N)$ converges to $0$ as $N \to \infty$, which is a weaker notion than the one we employ. We will compare the two definitions in Section \ref{sec-id}.
\end{rmk}

\begin{rmk}
Assumptions analogous to \ref{Theta-incoherence} and \ref{A-incoherence} are found in \cite{chen2019structured}. In Section \ref{sec-assumptions}, we will show that assumptions \ref{Theta-incoherence} and \ref{A-incoherence} are weaker than those found in \cite{chen2019structured}.
\end{rmk}

\begin{rmk}
Assumption \ref{C-bound} is slightly different than the boundedness assumption from \cite{chen2019structured}. Because we assume the supremum is strictly less than $C$, sufficiently small perturbations are permitted without violating the bound. This simplifies some of the analysis without changing the essential properties of the model.
\end{rmk}

Before stating the main results, we introduce the key concept of \emph{masking}, defined as follows.

\begin{definition}
We say $k'$ \emph{masks} $k$ if $\supp(Q_{k'}) \subset \supp(Q_k)$.
\end{definition}

\section{Main results}
\label{sec-main}

We provide necessary and sufficient conditions on the design matrix $Q$ which characterize when the latent factors $\Theta_k$ and the factor loadings $A_k$ are identifiable. Theorem \ref{id-theta} addresses identifiability of $\Theta_k$; the identifiability condition is similar to the one in \cite{chen2019structured}, but makes sense in a non-asymptotic setting; and the proof of Theorem \ref{id-theta} amounts to a tightening of the proof of Proposition 8 in \cite{chen2019structured}. Theorem \ref{id-a} characterizes identifiability of $A_k$, and appears to be new.

\begin{thm}
\label{id-theta}
For each $k$, $\Theta_k$ is identifiable if and only if $k$ does not mask any $k' \ne k$, or equivalently if 
\begin{align}
\label{eq-intersection}
\{k\} = 
\bigcap_{S \subset \{1,\dots,K\} \atop {k \in S \atop \RR(S) \text{ non-empty} }} S.
\end{align}
\end{thm}

\begin{thm}
\label{id-a}
Suppose $\supp(Q_k)$ is non-empty for all $k$. Then for each $k$, $A_k$ is identifiable if and only if no $k' \ne k$ masks $k$.
\end{thm}

\begin{rmk}
The identifiability conditions in Theorems \ref{id-theta} and \ref{id-a} may be efficiently checked for any given design matrix $Q$, simply by verifying that $\supp(Q_k)$ does not lie entirely within $\supp(Q_{k'})$ when $k' \ne k$.
\end{rmk}

\begin{rmk}

Clearly, if all $\Theta_k$ are identifiable -- that is, the entire matrix $\Theta$ is identifiable -- then so too is the entire matrix $A$; and vice versa. Theorems \ref{id-theta} and \ref{id-a} together provide more granular information on the relationship between $\Theta$ and $A$. Specifically, if $\Theta_k$ is \emph{not} identifiable, Theorem \ref{id-theta} states that there must be some $k' \ne k$ that is masked by $k$. Theorem \ref{id-a}, in turn, tells us that $A_{k'}$ is not identifiable. That is, knowing which column of  $\Theta$ is not identifiable automatically tells us what which column of $A$ is not identifiable. The same reasoning applies in reverse as well: knowing which column of  $A$ is not identifiable automatically tells us what which column of $\Theta$ is not identifiable.

\end{rmk}

\subsection{Technical lemmas}

\begin{lem}
\label{lem-independent}
Suppose $\supp(Q_{k'})$ is not empty. If $k'$ masks $k \ne k'$, then $A_k$ and $A_{k'}$ are linearly independent.
\end{lem}

\begin{proof}
Take any index $j \in \supp(Q_{k'})$, and let $S = \{k'':Q_{k''}(j) = 1\}$; then $j$ is contained in $\RR(S)$, and $k,k' \in S$. From assumption \ref{A-incoherence} $A_{[\RR(S),S]}$ has linearly independent columns; in particular $A_k(\RR(S))$ and $A_{k'}(\RR(S))$ are independent, and hence so too are $A_{k'}$ and $A_k$.
\end{proof}

\begin{lem}
\label{lem-independent2}
Suppose $k'$ masks $k \ne k'$, and let $\epsilon \in \R$. Then $\widetilde A = [A_1,\dots,A_{k-1},A_k + \epsilon A_{k'}, A_{k+1}, \dots, A_K]$ satisfies assumptions \ref{A-incoherence} and \ref{constraint-condition}.
\end{lem}

\begin{proof}
Assumption \ref{constraint-condition} is immediate, since the support of $A_k + \epsilon A_{k'}$ is still contained in $\supp(Q_k)$, because $\supp(A_{k'}) \subset \supp(A_k)$.

We now show that assumption \ref{A-incoherence} holds. Without loss of generality, take $k=K$ and $k'=1$; so $1$ masks $K$. Take any subset $S \subset \{1,\dots,K\}$, with $\RR(S)$ non-empty. We will show that the columns of $\widetilde A_{[\RR(S),S]}$ are linearly independent. This follows immediately from assumption \ref{A-incoherence} for $A$ if $K \notin S$; so assume $K \in S$.

First suppose $1 \in S$. From assumption \ref{A-incoherence} the vectors $A_k(\RR(S))$, $k \in S$, are linearly independent. Since $1$ and $K$ are in $S$, linear independence is preserved after replacing $A_K(\RR(S))$ with $A_K(\RR(S)) + \epsilon A_1(\RR(S))$.

Next, suppose $1 \notin S$. Then by definition $\supp(Q_1)$ is disjoint from $\RR(S)$, so $A_1(j) = 0$ for $j \in \RR(S)$. Consequently, $A_K(\RR(S)) = A_K(\RR(S)) + \epsilon A_1(\RR(S))$, and since $A_k(\RR(S))$, $k \in S$, are linearly independent, the same is true after replacing $A_K(\RR(S))$ by $A_K(\RR(S)) + \epsilon A_1(\RR(S))$.
\end{proof}

\begin{lem}
\label{lem-intersection}
Suppose $k$ does not mask any $k' \ne k$. Then
\begin{align}
\label{eq-intersection2}
\{k\} = 
\bigcap_{S \subset \{1,\dots,K\} \atop {k \in S \atop \RR(S) \text{ non-empty} }} S.
\end{align}
\end{lem}
\begin{proof}
Because $k$ does not mask any other $k'$, there must exist some subset $S \subset \{1,\dots,K\}$ containing $k$ with $\RR(S)$ non-empty. Indeed, $\supp(Q_k)$ must be non-empty, since otherwise $k$ would mask every $k'$. But each $j \in \supp(Q_k)$ is contained in $\RR(S)$, where $S = \{k'':Q_{k''}(j)=1\}$; and $k \in S$. Consequently, the right side of \eqref{eq-intersection2} is non-empty, and obviously contains $k$.

To show the reverse inclusion, take any $k' \ne k$. Since $k$ does not mask $k'$, $\supp(Q_{k}) \setminus \supp(Q_{k'})$ is non-empty. Each $j \in \supp(Q_{k}) \setminus \supp(Q_{k'})$ is contained in $\RR(S)$, where $S = \{k'':Q_{k''}(j)=1\}$ contains $k$ but not $k'$, implying that $k'$ is not contained in the right side of \eqref{eq-intersection2}.
\end{proof}

The converse to Lemma \ref{lem-intersection} is also true:

\begin{lem}
\label{lem-intersection2}
Suppose \eqref{eq-intersection2} holds. Then $k$ does not mask any $k' \ne k$.
\end{lem}

\begin{proof}
Without loss of generality, suppose $k=K$. If $\supp(Q_K)$ were empty (i.e.\ $Q_K(j)=0$ for all $j$), then for any $S \subset \{1,\dots,K\}$ containing $K$, $\RR(S) \subset \supp(Q_K)$ would also be empty, and the right side of \eqref{eq-intersection2} would be empty; a contradiction. Consequently, $\supp(Q_K)$ must be non-empty.

For contradiction, suppose without loss of generality that $K$ masks $1$; then $\supp(Q_K) \setminus \supp(Q_1)$ is empty. If $S \subset \{1,\dots,K\}$ contains $K$ but not $1$, then $\RR(S) \subset \supp(Q_K) \setminus \supp(Q_1)$, so $\RR(S)$ is also empty and $S$ is not included in the right side of \eqref{eq-intersection2}. Therefore, the only $S$ included on the right side of \eqref{eq-intersection2} contain both $K$ and $1$. But then $1$ is also in the intersection, a contradiction.
\end{proof}

\subsection{Proof of Theorem \ref{id-theta}}
First, suppose, without loss of generality, that $K$ masks $1$. We write:
\begin{align}
M = \Theta_1 A_1^T + \Theta_2 A_2^T + \dots + \Theta_K A_K^T
= \Theta_1 (A_1 + \epsilon A_K)^T  + \Theta_2 A_2^T + \dots 
    + (\Theta_K  - \epsilon \Theta_1 )A_K^T,
\end{align}
where $\epsilon$ is sufficiently small so as to not violate assumption \ref{C-bound}. From Lemma \ref{lem-independent2}, assumptions  \ref{A-incoherence} and \ref{constraint-condition} and are still satisfied by $A_1 + \epsilon A_K,A_2,\dots, A_K$. Assumption \ref{Theta-incoherence} still holds if we replace $\Theta_K$ by $\Theta_K  - \epsilon \Theta_1$. Since assumption \ref{Theta-incoherence} implies $\Theta_K - \epsilon \Theta_1$ and $\Theta_K$ are linearly independent, $\Theta_K$ is not identifiable.

For the other direction, assume that component $K$ does not mask any other component $k \ne K$. Suppose $M = \widetilde \Theta \widetilde A^T$ is another factorization of $M$ satisfying the model assumptions \ref{Theta-incoherence} -- \ref{C-bound}. We will show that $\Theta_K$ and $\widetilde \Theta_K$ are linearly dependent. 

Observe that $\supp(Q_K)$ is non-empty, since otherwise it would mask every $k$. Each $j \in \supp(Q_K)$ is contained in $\RR(S)$, where $S=\{k:Q_{k}(j) = 1\}$. Then if $k \notin S$ and $j \in \RR(S)$, we must have $A_k(j)=0$. Consequently, if $j \in \RR(S)$, $M_j(i) = \sum_{k=1}^{K} \Theta_k(i) A_k(j) =\sum_{k\in S} \Theta_k(i) A_k(j)$, and so we may write
\begin{align}
M_{[:,\RR(S)]} = \Theta_{[:,S]} (A_{[\RR(S),S]})^T.
\end{align}
By assumption \ref{A-incoherence}, $A_{[\RR(S),S]}$ has linearly independent columns, and since $\Theta$ has linearly independent columns, the column space of $M_{[:,\RR(S)]}$ has dimension $|S|$. Consequently, if we define $V_S \equiv \sp\{M_j : j \in \RR(S)\}$, then
\begin{math}
V_S = \sp\{\Theta_k : k \in S\}
\end{math}
and $\dim(V_S) = |S|$.

Because the $\Theta_k$ are linearly independent and $V_S = \sp\{\Theta_k : k \in S\}$, we have
\begin{math}
V_{S} \cap V_{S'} = V_{S \cap S'}.
\end{math}
Consequently
\begin{align}
\Theta_K \in V_{S_K}
= \bigcap_{S \subset \{1,\dots,K\} \atop {K \in S \atop \RR(S) \text{ non-empty} }} V_S 
\end{align}
where $S_K$ is the intersection of all sets $S$ with $K \in S$ and $\RR(S)$ non-empty. But because $K$ does not mask any $k \ne K$, Lemma \ref{lem-intersection} implies that $S_K = \{K\}$, and so $V_{S_K} = \sp\{\Theta_K\}$. But the exact same argument with $\widetilde \Theta$ and $\widetilde A$ in place of $\Theta$ and $A$ also shows $V_{S_K} = \sp\{\widetilde \Theta_K\}$. Consequently, $\Theta_K$ and $\widetilde \Theta_K$ are linearly dependent.
\subsection{Proof of Theorem \ref{id-a}}
First, let us suppose without loss of generality that $k=K$ is masked by $k'=1$. We write
\begin{align}
M = \Theta_1 A_1^T + \Theta_2 A_2^T + \dots + \Theta_K A_K^T
=(\Theta_1 - \epsilon\Theta_K) A_1^T + \Theta_2 A_2^T + \dots 
    + \Theta_K (A_K + \epsilon A_1)^T,
\end{align}
where $\epsilon$ is sufficiently small so as to not violate assumption \ref{C-bound}. From Lemma \ref{lem-independent2}, assumptions  \ref{A-incoherence} and \ref{constraint-condition} and are still satisfied by $A_1,A_2,\dots, A_K + \epsilon A_1$. Assumption \ref{Theta-incoherence} still holds if we replace $\Theta_1$ by $\Theta_1  - \epsilon \Theta_K$. From Lemma \ref{lem-independent}, $A_K + \epsilon A_1$ and $A_K$ are linearly independent. Consequently, $A_K$ is not identifiable.

For the other implication, suppose $M = \widetilde \Theta \widetilde A^T$ is another factorization within the same model, and that $\widetilde A_K$ and $A_K$ are linearly independent.  Let $\RR_K = \supp(Q_K)^c$ be the set of roots of $Q_K$; then $A_K$ and $\widetilde A_K$ are both zero on $\RR_K$.

Since the column space of $\widetilde A$ is contained in the column space of $A$, $\widetilde A_K$ is in the span of $A_1,\dots,A_K$. Therefore, there are coefficients $c_1,\dots,c_{K-1}$, not all zero, so that
\begin{align}
\label{eq-vanishing}
\sum_{k=1}^{K-1} c_k A_{k} (\RR_K) = \0.
\end{align}

Suppose, without loss of generality, that $c_1 \ne 0$. We will show that $1$ masks $K$. Suppose not; then $\supp(Q_1) \cap \RR_K = \supp(Q_1) \setminus \supp(Q_K)$ is non-empty. Take any $j \in \supp(Q_1) \cap \RR_K$; then $j$ is contained in $\RR(S)$, where $S=\{k:Q_{k}(j) = 1\}$. Since $K \notin S$, $\RR(S) \subset  \RR_K$. Furthermore, if $k \notin S$ and $j' \in \RR(S)$ then $A_k(j') = 0$. Hence from \eqref{eq-vanishing}
\begin{align}
\sum_{k \in S} c_k A_{k} (\RR(S)) = \0.
\end{align}
But by assumption \ref{A-incoherence}, the columns of $A_{[\RR(S),S]}$ are linearly independent; so we must have $c_k = 0$ for all $k \in S$. Since $1 \in S$, this contradicts that $c_1 \ne 0$.

\section{Discussion}
\label{sec-additional}

We conclude with a discussion comparing our work to \cite{chen2019structured}. In this section, we will treat $\Theta_k$ and $A_k$ as functions on $\Z_+$, rather than finite-length vectors, since this is the setting used in \cite{chen2019structured}. As noted in Remark \ref{rmk-asymptotic}, Theorems \ref{id-theta} and \ref{id-a} are valid in this doubly-asymptotic model. 

To aid the discussion, it is convenient to define the following notion.
\begin{definition}
A subset $\Delta \subset \Z_+$ is \emph{negligible} if
\begin{align}
\lim_{N \to \infty} \frac{|\Delta \cap \{1,\dots,N\}|}{N} = 0.
\end{align}
\end{definition}

In other words, $\Delta$ is negligible if the fraction of entries it contains from $\{1,\dots,N\}$ vanishes as $N$ grows.

\begin{rmk}
\label{rmk-negligible}
Any finite subset of $\Z_+$ is negligible. Furthermore, the definition of negligible depends crucially on the ordering of $\Z_+$. Indeed, if $\Delta$ is any infinite subset of $\Z_+$, we can always reorder $\Z_+$ so that $|\Delta \cap \{1,\dots,N\}|/N$ converges to a positive number, by interlacing the elements of $\Delta$ and $\Z_+ \setminus \Delta$. Similarly, we can reorder $\Z_+$ so that arbitrarily large gaps occur between the elements of $\Delta$, making $\Delta$ negligible under that ordering.
\end{rmk}

\subsection{Assumptions \ref{Theta-incoherence} and \ref{A-incoherence}}
\label{sec-assumptions}

In \cite{chen2019structured}, assumption \ref{Theta-incoherence} is replaced by the assumption that the limsup of the minimum singular values of the matrices $\Theta_{[1:N,1:K]} / \sqrt{N}$ is positive as $N \to \infty$; an analogous assumption is made in place of assumption \ref{A-incoherence}. The assumptions in \cite{chen2019structured} imply assumptions \ref{Theta-incoherence} and \ref{A-incoherence}. Indeed, suppose $B = [B_1,\dots,B_K]$, where each $B_k$ is a bounded function on $\Z_+$; and suppose that
\begin{align}
\label{limit-sv}
\limsup_{n \to \infty} \frac{\sigma_{\min}(B_{[1:n,1:K]})}{\sqrt{n}} > 0,
\end{align}
where $\sigma_{\min}$ denotes the smallest singular value. Then $B_1,\dots,B_K$ are linearly independent, since for sufficiently large $n$ the minimum singular value of $B_{[1:n,1:K]}$ must be positive.

It is not difficult to see that the converse statement is false; that is, assumptions \ref{Theta-incoherence} and \ref{A-incoherence} do not imply the corresponding assumptions from \cite{chen2019structured}. For example, we may take $\supp(B_1)$ to be the positive even integers, and $\supp(B_2)$ to be the positive odd integers; and define $B_1(2i) = 1/2i$ and $B_2(2i-1)=1/(2i-1)$. Then $B_1$ and $B_2$ are linearly independent. Take any large $n$ and $m < n$. Define $\TT_n = \{1,\dots,m\}$ and $\RR_n = \{m+1,\dots,n\}$, and partition $B^{(n)} \equiv B_{[1:n,1:2]} / \sqrt{n}$ into $B^{(n)}(\TT_n)$ and $B^{(n)}(\RR_n)$. Then the squared Frobenius norm of $B^{(n)}$ may be bounded above:
\begin{align}
\|B^{(n)}\|_F^2
= \|B^{(n)}(\TT_n)\|_F^2 + \|B^{(n)}(\RR_n)\|_F^2
\le \frac{2m}{n} + \frac{1}{m^2} \frac{2(n-m)}{n}.
\end{align}
Choosing $m=O(\sqrt{n})$ shows that the norm of $B^{(n)}$ converges to $0$ as $n \to \infty$, and so condition \eqref{limit-sv} is violated.

\subsection{Identifiability}
\label{sec-id}

As noted in Remark \ref{rmk-id}, \cite{chen2019structured} employs a weaker notion of identifiability of $\Theta_k$ than the one we use in the present work. In particular, the definition from \cite{chen2019structured} permits $\Theta_k$ and $\widetilde \Theta_k$ to differ (modulo a global rescaling) on negligible subsets of $\Z_+$.

As noted in Remark \ref{rmk-negligible}, any finite set is negligible, and any infinite subset may be made negligible or non-negligible by reordering $\Z_+$. Consequently, the definition of identifiability employed in \cite{chen2019structured} depends on the ordering of $\Z_+$. By contrast, the stronger notion of identifiability of $\Theta_k$ employed in the present work does not depend on a specified ordering.

\subsection{Condition \eqref{eq-intersection}}

Condition \eqref{eq-intersection} from Theorem \ref{id-theta} may be easily verified for any specified matrix $Q$. A similar condition appears in Theorem \ref{id-theta} from \cite{chen2019structured}, which we may state as follows:
\begin{align}
\label{eq-intersection3}
\{k\} = 
\bigcap_{S \subset \{1,\dots,K\} \atop {k \in S \atop \RR(S) \text{ non-negligible} }} S.
\end{align}
The right side of \eqref{eq-intersection3} is the intersection of all subsets $S \subset \{1,\dots,K\}$ containing $k$ where $\RR(S)$ are \emph{non-negligible}; by contrast, condition \eqref{eq-intersection} from Theorem \ref{id-theta} is the intersection of all such $S$ with $\RR(S)$ that are \emph{non-empty}. While the latter condition may be verified for finite-sized matrices $M$, the condition that $\RR(S)$ is non-negligible is an asymptotic condition, which is not determinable for a finite sized matrix. Furthermore, as noted in Remark \ref{rmk-negligible}, it depends on the ordering of the indices in $\Z_+$. While conceptually similar to \eqref{eq-intersection3}, the condition \eqref{eq-intersection} given in Theorem \ref{id-theta} is more suitable in practical settings as it is well-defined non-asymptotically.

\section*{Acknowledgements}

I am grateful to Xiaoou Li for discussing her work from \cite{chen2019structured}, and to the reviewers for their helpful comments. I acknowledge support from NSF award IIS-1837992 and BSF award 2018230.

\bibliographystyle{plain}
\bibliography{refs}

\end{document}